\documentclass[12pt]{amsart}
\usepackage[normalem]{ulem}
\usepackage{graphicx}
\usepackage{latexsym}
\usepackage{amsfonts}
\usepackage{amscd}
\usepackage{amssymb}
\usepackage{amsmath}
\usepackage{amsmath}
\usepackage{hyperref}
\usepackage{amssymb}
\usepackage{mathrsfs}
\usepackage{tikz}
\tikzstyle{vertex}=[circle]
\tikzstyle{goto}=[->,shorten >=1pt,>=stealth,semithick]
\usetikzlibrary{graphs}

\newtheorem{thm}{Theorem}[section]

\newtheorem{lemma}[thm]{Lemma}
\newtheorem{prop}[thm]{Proposition}
\newtheorem{cor}[thm]{Corollary}

\theoremstyle{definition}
\newtheorem{definition}[thm]{Definition}

\theoremstyle{remark}
\newtheorem{remark}[thm]{Remark}

\numberwithin{equation}{section}

\newcommand{\mc}{\mathcal}

\newcommand{\NN}{\mathbb{N}}
\newcommand{\CC}{\mathbb{C}}
\newcommand{\ZZ}{\mathbb{Z}}
\newcommand{\Q}{\mathcal{Q}}
\newcommand{\cO}{\mathcal{O}}

\newcommand{\Id}{\text{Id}}

\renewcommand{\L}{\Lambda}
\renewcommand{\a}{\alpha}
\renewcommand{\b}{\beta}
\newcommand{\dom}{\text{dom}}
\newcommand{\ran}{\text{ran}}

\newcommand{\iso}{\text{Iso}}
\newcommand{\M}{\mathcal{M}}
\newcommand{\I}{\mathcal{I}}

\newcommand{\J}{\mathcal{J}}

\newcommand{\B}{\mathcal{B}}
\renewcommand{\H}{\mathcal{H}}

\newcommand{\Ef}{\widehat E_0}

\newcommand{\Eu}{\widehat E_\infty}
\newcommand{\Et}{\widehat E_{\text{tight}}}
\newcommand{\Ct}{C^*_{\text{tight}}}
\newcommand{\Cr}{C^*_{\text{r}}}
\newcommand{\gt}{\mathcal{G}_{\text{t}}}

\newcommand{\vp}{\varphi}

\newcommand{\gtwo}{\g^{(2)}}
\newcommand{\go}{\g^{(0)}}
\newcommand{\cF}{\mathcal{F}}

\newcommand{\Iso}{\textup{Iso}}
\newcommand{\Ered}{E_{\text{red}}}
\newcommand{\g}{\mathcal{G}}
\newcommand{\C}{\mathcal{C}}
\newcommand{\spn}{\text{span}}
\newcommand{\Js}{J_{\text{sing}}}
\newcommand{\Cess}{C^*_{\text{ess}}}

\newcommand{\ga}{\gamma}
\newcommand{\de}{\delta}

\begin{document}
	
	\title{Uniqueness theorems for combinatorial C*-algebras}

	\author{Charles Starling}
	\address[C. Starling]{Carleton University, School of Mathematics and Statistics. 4302 Herzberg Laboratories}
	\email{cstar@math.carleton.ca}
	
	\thanks{C. Starling is funded by the NSERC discovery grant RGPIN-2021-03834.}

	\begin{abstract}
		Spielberg's construction of C*-algebras from left cancellative small categories is a common generalization for most C*-algebras one would consider to come from ``combinatorial data,'' including graph and $k$-graph C*-algebras, Li's semigroup C*-algebras, Nekrashevych's self-similar action algebras, and more. We use known groupoid models of these algebras and Exel's theory of tight representations of inverse semigroups to prove uniqueness theorems for these C*-algebras.
		
		 As applications, we improve on our previous uniqueness theorem for the boundary quotient C*-algebras of right LCM monoids, and we also generalize the uniqueness theorem of Brown, Nagy, and Reznikoff for row-finite higher-rank graphs to the finitely aligned case. 
	\end{abstract}
	
	\maketitle
	\section{Introduction}
 Broadly speaking, this paper concerns C*-algebra constructed from combinatorial data. Such constructions appear throughout the history of the subject, starting perhaps from the seminal work of Cuntz \cite{Cu77} who gave the first examples (the Cuntz algebras $\cO_n$) of simple C*-algebras generated by isometries. Discovery of these C*-algebras led to that of Cuntz-Krieger algebras associated to $\{0,1\}$-matrices \cite{CK80}, graph C*-algebras \cite{KPR98}, Exel-Laca algebras \cite{EL99}, and $k$-graph C*-algebras \cite{KuPa00}.
 
 Another generalization of the Cuntz algebras comes from semigroup theory. In his study of Wiener-Hopf operators, Nica \cite{Ni92} defined a C*-algebra from certain group-embeddable monoids and showed that when he applies his construction to the free semigroup on two elements, he recovers $\cO_2$. Li generalized this definition \cite{Li12} to general left-cancellative monoids.
 
 A common generalization to all these constructions was given by Spielberg \cite{Sp20} who defined a C*-algebra from a general {\em left-cancellative small category} (LCSC). Each of the above combinatorial objects are special cases of such: for example, a monoid is a small category with one object, while the set of finite paths through a graph or a $k$-graph is a small category without inverses. More examples of LCSC C*-algebras appear in the literature: Nekrashevych's self-similar group C*-algebras \cite{Nek09}, Exel and Pardo's generalization to self-similar actions on graphs \cite{EP17} and further generalizations to self-similar actions on $k$-graphs \cite{LY19,  LY21}.
 
 Before Spielberg's work, Exel \cite{Ex08} had already proposed inverse semigroups as the appropriate framework for combinatorial C*-algebras. These perspectives are of course closely related, as \cite{OP20} \cite{LiGarside1} show how to obtain an inverse semigroup from an LCSC and use its {\em tight representations} to characterize Speilberg's algebra. On the other hand, \cite{DGKNW20} shows that any inverse semigroup (up to Morita equivalence) arises from an LCSC. 
 
This paper is about uniqueness theorems for LCSC C*-algebras.  We first recall the original Cuntz-Krieger uniqueness theorem for graph C*-algebras to motivate our results. If $E$ is a row-finite directed graph, then its graph C*-algebra $C^*(E)$ is the universal C*-algebra generated by a set of partial isometries $\{s_e\}_{e\in E^1}$ indexed by the edges of $E$ and a set of mutually orthogonal projections $\{p_e\}_{e\in E^0}$ indexed by the vertices of $E$ subject to the relations
\begin{enumerate}
	\item[(CK1)] $s_e^*s_e = p_{d(e)}$, and
	\item[(CK2)] $\sum_{e\in r^{-1}(v)} s_es_e^* = p_v$ for all vertices $v\in E^0$
\end{enumerate} 

The original Cuntz-Krieger uniqueness theorem \cite{KPR98} states that as long as every cycle in $E$ has an entry, then any $*$-homomorphism $\vp:C^*(E)\to A$ into a C*-algebra $A$ is injective if and only if $\vp(p_v)\neq 0$ for all vertices $v$. Letting $D$ be the C*-algebra generated by the set $\{p_v\}_{v\in E^0}$, we could equivalently say
\begin{enumerate}
	\item a $*$-homomorphism $\vp:C^*(E)\to A$ is injective if and only if $\vp|_{D}$ is injective, or
	\item every nonzero ideal of $C^*(E)$ has nonzero intersection with $D$. 
\end{enumerate}
Given the second item, we say that $D\subseteq C^*(E)$ is an {\em ideal-detecting subalgebra} of $C^*(E)$, and any theorem that identifies an ideal-detecting subalgebra will be called a {\em uniqueness theorem}. The importance of such theorems for C*-algebras coming from generators and relations cannot be overstated, as injectivity of $*$-homomorphisms from such can be difficult to determine: we could not write a better explanation of this problem than that in the introduction to \cite{LS22}, so we direct the reader there.

The original Cuntz-Krieger uniqueness theorem was generalized by Reznikoff and Nagy \cite{NR12} who defined a subalgebra $\mc{M}(E)$ that they called the {\em abelian core} of $E$ that detects ideals in $C^*(E)$ without the assumption that every cycle in $E$ has an entry. These ideas were further developed in \cite{NR14} (defining abelian cores in other contexts), in \cite{BNR14} (generalized to row-finite $k$-graphs), and finally in \cite{BNRSW16}. 

The paper \cite{BNRSW16} is a crucial point in the story, as their uniqueness theorem is about subalgebras arising from the underlying \'etale groupoid. In a way that will be explained below, to an \'etale groupoid $\g$ one can construct a reduced C*-algebra $C^*_r(\g)$. Then \cite[Theorem~3.1]{BNRSW16} says that if $\g$ is a locally compact Hausdorff \'etale groupoid, then $C^*_r(\Iso(\g)^\circ)$ is an ideal-detecting subalgebra of $C^*_r(\g)$ ($\Iso(\g)^\circ$ is the interior of the isotropy subgroupoid of $\g$). This result was generalized to twisted groupoids by Armstrong in \cite{Arm22} and to more general families of subgroupoids by Goerke, Eagle, and Laca in \cite{GEL24}.

If a C*-algebra $A$ is isomorphic to $C^*_r(\g)$ for some locally compact \'etale groupoid $\g$, we say $\g$ is a {\em groupoid model} for $A$. As many of the combinatorial C*-algebras mentioned above are known to have groupoid models, the results of \cite{BNRSW16} can be applied to them. Following Lalonde and Milan \cite{LM17}, this is what we more-or-less did in \cite{St22}. Specifically we:
\begin{enumerate}
	\item Generalized \cite[Theorem~3.1]{BNRSW16} to certain open subgroupoids $\mc{F}\subseteq \Iso(\g)$ when $\g$ is Hausdorff \cite[Theorem~2.1]{St22},
	\item applied this result to inverse semigroups $S$ whose C*-algebras have Hausdorff groupoid models \cite[Theorem~3.3]{St22},
	\item then applied the inverse semigroup result to C*-algebras of right LCM monoids that have Hausdorff groupoid models \cite[Theorem~4.1]{St22} \cite[Theorem~2.1]{BS24}.
\end{enumerate}

In the above list, the Hausdorff assumption is conspicuous. One gets the nicest results in the Hausdorff case, but there are many interesting examples of combinatorial C*-algebras whose underlying groupoids are non-Hausdorff. A major complication in proving uniqueness theorems in the non-Hausdorff case is the {\em singular ideal}. This ideal was defined in \cite{CEPSS} (see also \cite{KM21}, \cite{EPBook22}) and can be explained as follows: $C^*_r(\g)$ is the completion of a certain $*$-algebra of functions $\mc{C}(\g)$, and each function in $\mc{C}(\g)$ is continuous if and only if $\g$ is Hausdorff. Roughly speaking, if elements of $C^*_r(\g)$ can be viewed as functions, the singular ideal $\Js(\g)$ contains the functions that only take nonzero values on meagre sets (see below for the proper definition). 

We do not address the problem of determining when the singular ideal vanishes here (though recent work of Hume \cite{H25} provides a potential path for future work on this). Instead, we sidestep this issue entirely and say a subalgebra $B\subseteq C^*_r(\g)$ is {\em essentially ideal-detecting} if any ideal $I$ of $C^*_r(\g)$ that has zero intersection with $B$ satisfies $I\subseteq \Js(\g)$.

We now provide an overview of what follows. The present paper can be seen as a continuation of \cite{St22}, or perhaps more accurately as an expanded version of it. We follow the same pattern here:
\begin{enumerate}
	\item {\bf Groupoid result:} Recent advances in \cite{KKLRU21}  \cite{CN24} allow us to improve \cite[Theorem~2.1]{St22} by removing the Hausdorff condition, see Proposition~\ref{prop:CN47generalization} below.
	\item {\bf Inverse semigroup result:} This allows us to remove the Hausdorff assumption in \cite[Theorem~3.3]{St22} to obtain an improvement: if $S$ is an inverse semigroup, there is a subsemigroup $S^\Iso\subseteq S$ such that in the reduced tight C*-algebra $C^*_r(\gt(S))$ the subalgebra generated by $S^\Iso$ is essentially ideal-detecting; see Theorem~\ref{thm:MainThm1} for the precise statement.
	\item {\bf LCSC result:} By using an inverse semigroup picture for the C*-algebra of an LCSC given by Ortega and Pardo \cite{OP20}, we identify two essentially ideal-detecting subalgebras of LCSC algebras. See Theorem~\ref{thm:main_theorem} for the precise statement.
\end{enumerate}

We have two main applications of our results. 

\begin{enumerate}
	\item {\bf Right LCM monoids:} A monoid $P$ is called {\em right LCM} if it is left cancellative and, for any $p,q\in P$, $pP\cap qP$ is either empty or equal to $rP$ for some $r\in P$. The {\em core} of $P$ is the submonoid $P_0 = \{p\in P: pP\cap qP\neq \emptyset \text{ for all }q\in P\}$. Its reduced boundary quotient $\mc{Q}_r(P)$ is a certain C*-algebra generated by a copy of $P$, and we show below (Corollary~\ref{cor:LCM}) that the subalgebra generated by $P_0$ is essentially ideal-detecting in $\mc{Q}_r(P)$.
	\item {\bf $k$-graphs: } A $k$-graph is a LCSC $\L$ with a degree map $\deg:\L\to \NN^k$ such that $\deg(\a\b) = \deg(\a) + \deg(\b)$ and has unique factorization in the sense that if $\deg(\a) = m+n$ for some $m,n\in\NN^k$, then there exist unique $\b, \ga\in\L$ with $\a = \b\ga$, $\deg(\b) = m$ and $\deg(\ga) = n$. These properties imply the underlying groupoid is Hausdorff and amenable, so $\cO(\L)$ is the universal C*-algebra for partial isometries $\{W_\a\}_{\a\in\L}$ satisfying certain conditions, see Remark~\ref{rmk:CK_LCSC} below. In \cite{BNR14} they study the case where $\L$ is row-finite, and they define what it means for $(\a,\b)\in\L\times\L$ to be a {\em cycline pair}. In their main theorem, show that the subalgebra of $\cO(\L)$ generated by $\{W_\a W_\b^*: (\a,\b)\text{ is cycline}\}$ is ideal-detecting. In Theorem~\ref{thm:k_graph} we show that this subalgebra is ideal-detecting in the finitely-aligned case as well.
\end{enumerate} 

	{\bf Acknowledgement:} I thank Chris Bruce and Kevin Aguyar Brix for comments on an earlier draft of this work.
	
	\section{Groupoids and their C*-algebras}
	We recall the definitions around \'etale groupoids and the construction of their C*-algebras. For solid references on this, the reader is directed to \cite{R80} and \cite{Sim20}.
	
	A {\em groupoid} is a set $\g$ with a partially-defined multiplication which is associative where defined and for which each element has an inverse. Specifically, there is an inverse map $\ga\mapsto\ga^{-1}$ from $\g$ to $\g$ and a set $\gtwo\subseteq \g\times \g$ called the set of {\em composable pairs} and a product map $(\a,\b)\mapsto \a\b$ from $\gtwo$ to $\g$ such that
	\begin{enumerate}
		\item $(\a,\b),(\b,\ga)\in \gtwo$ implies $(\a\b)\ga = \a(\b\ga)$, 
		\item  $(\ga,\ga^{-1}), (\ga^{-1},\ga)\in\gtwo$ for all $\ga$ and $\ga^{-1}\ga\a = \a$, $\b\ga\ga^{-1} = \b$ whenever $(\ga, \a), (\b,\ga)\in\gtwo$, and
		\item $(\ga^{-1})^{-1} = \ga$ for all $\ga\in\g$. 
	\end{enumerate} The set $\go := \{\gamma\gamma^{-1}: \ga\in \g\}$ is called the {\em unit space} of $\g$. The maps $r, d:\g\to \go$ given by $r(\ga) = \ga\ga^{-1}$ and $d(\ga) = \ga^{-1}\ga$ are called the {\em range} and {\em source}\footnote{Usually the source map is denoted $s$. We use the variable $s$ so often for a general element of an inverse semigroup, that it is best for us to use $d$ for the source map.} maps respectively.  For any subset $\mc{L}\subseteq\g$ and $x\in \go$, we write
	\[
	\mc{L}^x = r^{-1}(x)\cap \mc{L}, \hspace{1cm} \mc{L}_x = d^{-1}(x)\cap \mc{L}, \hspace{1cm}\mc{L}^x_x = \mc{L}_x\cap \mc{L}^x.
	\]
	
	A pair $(\a,\b)$ is in $\gtwo$ if and only if $r(\b) = d(\a)$. A subset $U\subseteq \g$ is called a {\em bisection} if $r|_U$ and $d|_U$ are injective. 
	
	A groupoid is called {\em\'etale} if it is a topological groupoid (meaning it has a topology in which the inverse and product maps are continuous) such that $\go$ is Hausdorff in the relative topology and the range map is a local homeomorphism. An \'etale groupoid admits a basis of open bisections. If $U, V$ are (open) bisections, then so are $UV = \{\a\b:\a\in U, \b\in V, (\a,\b)\in \g^{(2)}\}$ and $U^{-1} : = \{\g^{-1}: \g\in U\}$. If $\g$ is \'etale, then $\g_x$ and $\g^x$ are discrete subspaces of $\g$. If $\g$ is \'etale and admits a basis of {\em compact} open bisections, it is called {\em ample}. 
	
	Let $\g$ be a locally compact \'etale groupoid. Any bisection $U\subseteq \g$ is homeomorphic to an open subset of $\go$, and hence is locally compact Hausdorff. For $f\in C_c(U)$, $f$ is viewed as a function on $\g$ by extending it to be 0 outside of $U$; note that when $\g$ is non-Hausdorff this function might not be continuous. Let
	\[
	\C(\g) = \spn \{f\in C_c(U) : U\subseteq \g\text{ is an open bisection}\}.
	\]
	
	The set of functions $\C(\g)$ was defined by Connes \cite{Co82} \cite{CoNCG} and is a $*$-algebra in a natural way: addition and scalar multiplication are defined pointwise, while if $U$ and $V$ are open bisections and $f\in C_c(U)$, $g\in C_c(V)$, then we let $fg$ be zero outside of the open bisection $UV$, and for $\ga\in UV$ we let
	\[
	fg(\ga) = f(\a)g(\b)
	\]
	where $(\a,\b)$ is the unique element of $U\times V$ such that $\a\b = \ga$ (uniqueness follows from the fact that $U$ and $V$ are bisections). This product is extended to $\C(\g)$ linearly. For $f\in C_c(U)$ define $f^*\in  C_c(U^{-1})$ by
	\[
	f^*(\ga) = \overline{f(\ga^{-1})}.
	\]
	We then extend this to $\C(\g)$ conjugate linearly. These operations make $\C(\g)$ a $*$-algebra. 
	
	For each $x\in \go$, let $\rho_x: \C(\g)\to \mc{B}(\ell^2(\g_x))$ denote the $*$-representation of $\C(\g)$ given by
	\[
	\rho_x(f)\delta_\ga = \sum_{\a\in \g_{r(\ga)}}f(\a)\delta_{\a\ga}.
	\]
	Then the {\em reduced C*-algebra of} $\g$, denoted $C_r^*(\g)$, is the completion of the image of $\C(\g)$ under the representation $\bigoplus_{x\in \go}\rho_x$. There is another C*-algebra associated to $\g$, called the {\em full C*-algebra of $\g$}, denoted $C^*(\g)$, which is universal for representations of $\C(\g)$. One defines a norm on $\C(\g)$ by $\|f\|: = \sup\{\|\pi(f)\|:\pi\text{ is a representation of }\C(\g)\}$, then $C^*(\g)$ is the completion of $\C(\g)$ in this norm. It is universal in the sense that if $\pi:\C(\g)\to A$ is a $*$-homomorphism of $\C(\g)$ into a C*-algebra $A$, there is a $*$-homomorphism $\vp:C^*(\g)\to A$ that agrees with $\pi$ on $\C(\g)$. 
	 
	 {\em Renault's $j$ map} is the injective linear contraction $j: C^*_r(\g)\to \ell^\infty(\g)$ given by
	 \[
	 j(a)(\ga) = \langle \rho_{d(\ga)}(a)\delta_{d(\ga)}, \delta_{\ga} \rangle \hspace{1cm}\ga\in \g, a\in C^*_r(\g)
	 \]
	 
	 The following description of the singular ideal is given in \cite{KKLRU21}. For a locally compact Hausdorff space $X$, recall the following notation:
	 \begin{itemize}
	 	\item $\B^\infty(X)$ = the C*-algebra of bounded Borel functions (sup norm),
	 	\item $\M^\infty(X)$ = the ideal of $\B^\infty(X)$ of functions with meagre support,
	 	\item Dix$(X) = \B^\infty(X)/\M^\infty(X)$, the Dixmier algebra of $X$.
	 \end{itemize}
	
	 The linear map $\Ered: \C(\g)\to \text{Dix}(\go)$ determined by $\Ered(f) = \left. f\right|_U$, where $U$ is any open dense subset of $\go$ on which $f$ is continuous, defines a local conditional expectation on $C^*_r(\g)$. The ideal of singular elements can then be defined by 
	\[
	\Js(\g): = \{a\in C^*_r(\g): \Ered(a^*a) = 0\}
	\]
	and the {\em essential C*-algebra of $\g$ is}
	\[
	\Cess(\g) : = C^*_r(\g)/\Js(\g).
	\]
	By \cite[Proposition~7.18]{KM21}, when $\g$ is covered by countably many bisections (as is the case for all the groupoids in this paper) $a\in \Js(\g)$ if and only if $d(\text{supp}(j(a)))$ is a meagre subset of $\go$. 
	
	We now define the concept we will be investigating in this paper. 
	\begin{definition}\label{def:ess_ideal_detecting}
		Let $\g$ be a locally compact \'etale groupoid. We say that a subalgebra $B\subseteq C^*_r(\g)$ is {\em essentially ideal-detecting} if any $*$-homomorphism $\pi:C^*_r(\g)\to A$ into a C*-algebra $A$ that is injective on $B$ satisfies $\ker\pi\subseteq \Js(\g)$. Equivalently, $B$ is essentially ideal-detecting if whenever an ideal $I$ of $C^*_r(\g)$ has zero intersection with $B$, then $I\subseteq\Js(\g)$. 
	\end{definition}
	
		In what follows we will be applying the following result (which is a generalization of \cite[Theorem~4.7]{CN24}) to specific classes of groupoids.
	
	\begin{prop}\label{prop:CN47generalization}
		Let $\g$ be a locally compact \'etale groupoid and suppose that $\g$ is either Hausdorff or can be covered by countably many open bisections. Suppose $\H\subseteq \g$ is an open subgroupoid of $\g$ with the property that the set $Y = \{x\in \go: \g^x_x = \H_x^x\}$ is dense. Then for any ideal $I$ of $C^*_r(\g)$ we have that $I\cap C^*_r(\H)\subseteq \Js(\g)$ implies $I\subseteq \Js(\g)$. 
	\end{prop}
	Note: Proposition~\ref{prop:CN47generalization} has the same assumptions as \cite[Theorem~4.7]{CN24}, and there Christensen and Neshveyev conclude that any ideal with 0 intersection with $C^*_r(\H)$ must be contained in $\Js(\g)$. So this is a slight improvement, as we relax the assumption that $C_r^*(\H)\cap I = \{0\}$ to $C_r^*(\H)\cap I\subseteq\Js(\g)$. This generalization will be needed to prove part of Theorem~\ref{thm:main_theorem}.
	\begin{proof}
		By \cite[Section 4]{CN24} there exists an open dense subset $X\subseteq \go$ such that 
		\begin{equation}\label{JsingDef}
			\Js(\g) = \bigcap_{x\in X}\ker(\rho_x)
		\end{equation}
		where $\rho$ denotes the left regular representation of $C_r^*(\g)$ at the unit $x$. Furthermore, this equality remains true if $X$ is replaced with any subset of $X$ that is also dense in $\go$. Because $Y = \{x\in \go: \g^x_x = \H_x^x\}$ is dense in $\go$ and $X$ is open and dense, we have that $X\cap Y$ is dense in $\go$, and so 
		\begin{equation}\label{JsingXY}
			\Js(\g) = \bigcap_{x\in X\cap Y}\ker(\rho_x).
		\end{equation}
		Suppose $a\in J:= C_r^*(\H)\cap I$ and $x\in X\cap Y$. Then $J\subseteq \Js(\g)$ and \eqref{JsingXY} imply that $j(a)(\gamma) = 0$ for all $\gamma\in \g^x\cup \g_x$; in particular $j(a)(\gamma) = 0$ for all $\gamma\in \g_x^x = \H^x_x$. Hence $\left.j(a)\right|_{\H^x_x} = 0$. Thus (using the notation of \cite[Proposition~4.6]{CN24}) we have that $J_x = 0$. By that same proposition we then have that $I_x \cap C^*_r(\H_x^x) = 0$. But since $C^*_r(\H_x^x) = C^*_r(\g_x^x)$ and $I_x \subseteq C^*_r(\g_x^x)$ we get that $I_x = 0$ for all $x\in X\cap Y$. Since $X\cap Y\subseteq \go$ is dense, \cite[Lemma~4.1]{CN24} implies that $I\subseteq \Js(\g)$.  
	\end{proof}
	
	\section{The tight C*-algebra of an inverse semigroup}
	
	\subsection{Inverse semigroups}
	
	A semigroup $S$ is an {\em inverse semigroup} if for all $s\in S$ there exists unique $s^*\in S$ such that $ss^*s = s$ and $s^*ss^* = s^*$. A {\em zero} in an $S$ is an element $0\in S$ such that $0s = s0 = 0$ for all $s\in S$. If $S$ has a zero it is necessarily unique. An element $e\in S$ is called an {\em idempotent} if $e^2 = e$. It is well-known that idempotents in an inverse semigroup commute. 
	
	Let $S$ be an inverse semigroup with idempotent set $E\subseteq S$ containing a zero element 0. Then $S$ carries a canonical order $s\leqslant t \iff ts^*s = s$. On $E$ this order becomes $e\leqslant f \iff ef = e$. With this order, $E$ is a semilattice. The order is preserved by multiplication, i.e. 
	\[
	s\leqslant t \implies sr\leqslant tr\text{ and }rs\leqslant rt
	\]
	for $r,s,t\in S$. 
	
	Two elements $s,t\in S$ are called {\em compatible} if $s^*t$ and $st^*$ are idempotents---note that idempotents are always compatible. For any subset $A\subseteq S$, we write $\bigvee_{a\in A}a$ for the least upper bound of $A$ in $S$ in the above canonical order, if it exists. By \cite[Lemma~ 1.1.6]{La98}, $\bigvee_{a\in A}a\in S$ implies that elements of $A$ are pairwise compatible. We say $S$ is {\em (resp. finitely) complete} if whenever $A$ is a (resp. finite) subset of $S$ whose elements are pairwise compatible, then $\bigvee_{a\in A}a\in S$. In addition, we say that $S$ is {\em (resp. finitely) distributive} if for any $s\in S$ and and (resp. finite) subset $A\subseteq S$, then 
	\[
	\bigvee_{a\in A}a\in S \implies s\left(\bigvee_{a\in A}a\right) = \bigvee_{a\in A}sa\in S
	\]
	
	We record some facts about the order in an inverse semigroup for use later.
	\begin{lemma}\label{lem:ISG_order_wedge}
		Let $S$ be a (resp. finite) distributive inverse semigroup, let $A\subseteq  S$ be a (resp. finite) subset of $S$ and suppose $s = \bigvee_{a\in A}a\in S$ .
		
		\begin{enumerate}
			\item If $e\leqslant a^*a$ for some $a\in S$, then $ses^* = aea^*$. 
			\item If $t\in S$ and $tt^*\leqslant aa^*$ for some $a\in A$, then $t^*st = t^*at$. 
		\end{enumerate}	
	\end{lemma}
		\begin{proof}
		(1) For any $b, c\in A$ we calculate
		\[
		bec^* = ba^*aea^*ac^*\leqslant aea^*
		\]
		as $ba^*, ac^*$ are idempotents. Hence $ses^* = \bigvee_{b,c\in A}bec^* \leqslant aea^*$ (while $\geqslant$ is automatic). 	
		
		(2) For $b\in A$, we have $aa^*b\leqslant a$ as $a^*b$ is an idempotent. Multiplying on the left by $t^*$ and the right by $t$ gives
		\[
		t^*at\geqslant t^*aa^*bt \geqslant t^*tt^*bt = t^*bt.
		\]
		Thus $t^*st = \bigvee_{b\in A}t^*bt\leqslant t^*at$ (while again $\geqslant$ is automatic).
	\end{proof}
	For any set $X$, let $\I(X)$ denote the set of all bijections between subsets of $X$. It is an inverse monoid when given the operation of composition on largest possible domain, and is called {\em symmetric inverse monoid on $X$}. It has a zero element, namely the empty function. For $f\in \I(X)$, we write $f:\dom(f)\to \ran(f)$.  Two functions $f,g\in \I(X)$ are compatible if and only if $f(x) = g(x)$ for all $x\in \dom(f)\cap \dom(g)$. If $F$ is a set of pairwise compatible elements of $\I(X)$, their join $\bigvee_{f\in F}f$ is the union $\bigcup_{f\in F}f$. 
	
	By \cite[Proposition~1.2.1]{La98} $\I(X)$ is complete and distributive, but we caution that these properties do not necessarily pass to subsemigroups. 
	
	{\bf For the rest of this paper, every inverse semigroup is assumed to have a zero element}.

	A {\em filter} in the idempotent semilattice $E$ is a proper subset $\xi\subsetneq E$ that is downwards directed (i.e. closed under products) and upwards closed (i.e. $e\in\xi$ and $e\leqslant f$ implies $f\in \xi$). The {\em spectrum} of $E$ is the set $\Ef$ of all filters equipped with the relative topology when viewed as a subspace of the product space $\{0,1\}^E$. For finite $X,Y\subseteq E$, the sets
	\[
	U(X,Y) : = \{\xi\in\Ef: x\in \xi\text{ for all }x\in X, Y\cap \xi =\emptyset\}
	\]
	form a compact open basis for this topology on $\Ef$.
	
	\begin{definition}
		Let $\xi\in \Ef$. We say that $\xi$ is
		\begin{enumerate}
			\item an {\em ultrafilter} if it is not properly contained in any other filter, and let $\Eu$ denote the set of ultrafilters;
			\item a {\em tight filter} if $\xi\in \overline{\Eu}$, and
			\item a {\em prime filter} if whenever $\bigvee_{e\in F}e\in \xi$ for some finite set $F\subseteq E$, then $F\cap \xi\neq \emptyset$. 
		\end{enumerate}
	\end{definition} The space $\Et :=\overline{\Eu}$ is called the {\em tight spectrum} of $E$. It is possible to give a characterization of tight filters: for finite subsets $X,Y\subseteq E$ define
	\[
	E^{X,Y}: = \{e\in E: e\leqslant x\text{ for all }x\in X\text{ and }ey = 0\text{ for all }y\in Y\}.
	\]
	We say that $C$ is a {\em cover} of $D$, for $C\subseteq D\subseteq E$, if for all $d\in D$ there exists $c\in C$ such that $cd\neq 0$. Then by \cite[Theorem~12.9]{Ex08}, $\xi$ is tight if and only if whenever $\xi\in U(X, Y)$ for finite sets $X,Y\subseteq E$, then every finite cover $Z$ of $E^{X,Y}$ has nonempty intersection with $\xi$. 
	The first statement of the following was proven in the LCSC case in \cite[Lemma~3.10]{OP20}. 
	
	\begin{lemma}\label{lem:distributive_prime_tight} Let $S$ be an inverse semigroup with idempotent semilattice $E$.
		\begin{enumerate}
			\item 	If $S$ is finitely distributive, then every tight filter in $E$ is a prime filter.
			\item If $S$ is finitely complete, every prime filter in $E$ is a tight filter.
		\end{enumerate}
	\end{lemma}
	\begin{proof}
		The proof of \cite[Lemma~3.10]{OP20} carries over almost verbatim in this case, but we give a recap. Let $\xi\in \Et$, let $F\subseteq E$ be a finite set, and suppose $c = \bigvee_{e\in F}e \in \xi$. Then $\xi\in U(\{c\},\emptyset)$, and $E^{\{c\},\emptyset} = \{f\in E: f\leqslant c\}$. Then if $0\neq f\leqslant c$,
		\[
		0\neq f = fc = \bigvee_{e\in F}fe
		\]
		implies that $fe\neq 0$ for some $e\in F$, so that $F$ is a finite cover of $E^{\{c\},\emptyset}$. Then $F\cap \xi\neq\emptyset$ by the above discussion, implying $\xi$ is prime.
		
		If $S$ is finitely complete and $\xi$ is prime, suppose that $X,Y\subseteq E$ are finite, $\xi\in U(X,Y)$ and that $Z$ is a finite cover of $E^{X,Y}$. Then $c : = \bigvee_{z\in Z}z$ is in $E$ by finite completeness, and is in $\xi$ by upwards closure. Because $\xi$ is prime, $Z\cap \xi\neq\emptyset$, so \cite[Theorem~12.9]{Ex08} implies $\xi$ is tight.
	\end{proof}
	Any inverse semigroup acts on its spectrum. Let $s\in S$ and suppose $\xi\in\Ef$ contains $s^*s$. Then 
	\[
	\theta_s(\xi) := \{ses^*:e\in\xi\}^\leqslant
	\]
	is a filter that contains $ss^*$. Then $\theta_s$ is a homeomorphism from $U(s^*s, \emptyset)$ to $U(ss^*, \emptyset)$, and so $\{\theta_s\}_{s\in S}$ is an action of $S$ on $\Ef$. This action leaves the tight spectrum invariant. In this paper we only deal with this action, so we will let
	\[
	D_e:= U(\{e\}, \emptyset)\cap\Et,
	\]
	and note that as above these are compact open subsets of $\Et$. 
	
	As in \cite[Lemma~3.1]{St22} and \cite[Definition~5.8]{LiGarside1} we let
	\begin{equation}\label{eq:Siso_def}
	S^\Iso = \{s\in S: e\leqslant s^*s\implies ses^*e\neq 0\}.
	\end{equation}
	This is an inverse subsemigroup of $S$ containing the idempotents. Then \cite[Lemma 4.9]{EP16} implies
	\[
	s\in S^\Iso \iff \theta_s(\xi) = \xi\hspace{1cm} \text{ for all }\xi\in D_{s^*s}.
	\] 
	
	Write $S\times_\theta \Et: = \{(s, \xi)\in S\times \Et: \xi\in D_{s^*s}\}$ and put an equivalence relation on $S\times_\theta \Et$ by saying $(s,\xi)\sim (t, \eta)$ if and only if $\xi = \eta$ and $se = te$ for some $e\leqslant s^*s, t^*t$ with $e\in \xi$. Then the {\em tight groupoid} of $S$, denoted $\gt(S)$, is the set of equivalence classes
	\[
	\gt(S): = \{[s, \xi]: s\in S, \xi\in D_{s^*s}\}
	\]
	with range, source, product, and inverse given by
	\[
	d([s,\xi]) = \xi, \hspace{.8cm}r([s,\xi]) = \theta_s(\xi)\hspace{.8cm}\hspace{.8cm}[t,\theta_s(\xi)][s,\xi] = [ts,\xi], \hspace{1cm}[s,\xi]^{-1} = [s^*, \theta_s(\xi)]
	\]
	where we are identifying the point $\xi\in \Et$ with the unit $[e,\xi]$ where $e$ is any element of $\xi$. 
	It is an ample groupoid when given the topology generated by sets of the form
	\[
	[s,U]:= \{[s,\xi]:\xi\in U\}, \hspace{1cm}U\subseteq D_{s^*s} \text{ compact open}.
	\]
	In particular, $\gt(S)$ is ample. 
	
	In \cite{Ex08} Exel defined the tight C*-algebra of $S$ as the universal C*-algebra for a class of representations of $S$, which he called {\em tight representations}. By \cite[Corollary~2.3]{DM14}, \cite{Ex21} we can define these representations in a simpler way than originally defined. A {\em representation} of $S$ in a C*-algebra $A$ is a multiplicative map $\pi:S\to A$ that sends $0$ to $0$, and a representation is called {\em cover-to-join} if whenever $e\in E$ and $C$ is a {\em finite} cover of $\{e\}$, then $\bigvee_{c\in C} \pi(c)$ exists in $A$ and is equal to $\pi(e)$. The {\em tight C*-algebra of $S$} is then, by \cite[Corollary~2.5]{DM14} the universal C*-algebra for cover-to-join representations of $S$, and is denoted $\Ct(S)$. 
	
	Exel showed that $\Ct(S)$ is isomorphic to $C^*(\gt(S))$, the full C*-algebra of the tight groupoid of $S$.  Write
	\begin{equation}\label{eq:Ts_elt_def}
	T_s:= 1_{[s,D_{s^*s}]} \in \C(\gt(S)).
	\end{equation}
	
	We recall some notation from Spielberg (see the paragraph above \cite[Theorem~6.3]{Sp14}. If $A$ is a C*-algebra and $F\subseteq A$ is a finite set of partial isometries whose initial and final projections form a commuting family. Suppose further that for all $s,t\in F$, we have $st^*t = ts^*s$. Then there is a partial isometry, denoted there by $\bigvee_{s\in F}s$, with initial projection $\bigvee_{s\in F}s^*s$ and final projection $\bigvee_{s\in F}ss^*$. Using the usual formula $p\vee q = p + q - pq$ for the join of two commuting projections, one can show that for partial isometries $s,t$ as above we have
	\begin{equation} \label{eq:operator_s_vee_t}
	s\vee t = s + t - ts^*s = s + t - st^*t.
	\end{equation}

	As pointed out in \cite{St22}, the fact that $S^\Iso$ contains the idempotent semilattice of $S$ implies that $C^*_r(\gt(S^\Iso))$ is isomorphic to the subalgebra of $C^*_r(\gt(S))$ generated by $\{T_s\}_{s\in S^\iso}$. Using Lemma \ref{prop:CN47generalization} we obtain the following strengthening of \cite[Theorem~3.4]{St22}.
	
	\begin{thm}\label{thm:MainThm1}
				Let $S$ be a countable inverse semigroup and let $S^\Iso$ be as in \eqref{eq:Siso_def}. Then if $I$ is an ideal of $C^*_r(\gt(S))$ such that $I\cap C^*_r(\gt(S^\Iso))\subseteq \Js(\gt(S))$, then $I\subseteq \Js(\gt(S))$. In particular, the subalgebra of $C^*_r(\gt(S))$ generated by $\{T_s\}_{s\in S^\iso}$ is essentially ideal-detecting.
	\end{thm}
	\begin{proof}
		Since $S$ is countable, $\gt(S)$ can be covered by countably many bisections. By \cite[Proposition 3.3]{St22}, $\gt(S^\iso)$ satisfies the conditions of Proposition~\ref{prop:CN47generalization}. The result follows.
	\end{proof}
	\begin{cor}
		Let $S$ be an inverse semigroup and suppose the singular ideal of $\Cr(\gt(S))$ vanishes. Then a $*$-homomorphism $\vp: \Cr(\gt(S)) \to B$ to a C*-algebra $B$ is injective if and only if it is injective on the C*-algebra generated by the set $\{T_s\}_{s\in S^\iso}$ \eqref{eq:Ts_elt_def}.
	\end{cor}

	\section{Left cancellative small categories and associated inverse semigroups}
	
	A {\em small category} is a quadruple $(\L, \L^0, r,d)$ where $\L^0\subseteq\L$ are sets and $r, d:\L\to \L^0$ are maps called the {\em range} and {\em source} maps respectively (onto the set $\L^0$ of {\em objects}) satisfying the following:
	\begin{enumerate}
		\item $r(x) = d(x) = x$ for all $x\in \L^0$,
		\item there is a composition map from $\L^2 = \{(\a, \b)\in \L\times\L: d(\a) = r(\b)\}$ to $\L$ with $(\a,\b)\to \a\b$ such that $r(\a\b) = r(\a)$ and $d(\a\b) = d(\b)$,
		\item $r(\a)\a = \a = \a d(\a)$ for all $\a\in\L$, and
		\item $\a(\b\gamma) =(\a\b)\gamma$ whenever the compositions on both sides are defined.
	\end{enumerate}   
	We say $\a\in\L$ is {\em invertible} if there exists $\b\in\L$ such that $\a\b = r(\a)$. The set of all invertible elements of $\L$ is denoted $\L^{-1}$. 
	
	We say $\L$ is {\em left (resp. right) cancellative} if $\a\b = \a\gamma \implies \b = \gamma$ (resp. $\a\b = \gamma\b \implies \a =\gamma$) for all $\a,\b,\gamma\in \L$, and say $\L$ is cancellative if it is both left and right cancellative. 
	
	If $\L$ is left cancellative, then if $\a\in\L^{-1}$, the element $\b$ such that $\a\b = r(\a)$ is unique, is denoted $\a^{-1}$, and also satisfies $\a^{-1}\a = d(\a)$. 
	
	In what follows, we will be primarily dealing with left cancellative small categories, and for these we will use the abbreviation LCSC.
	
	\begin{remark}
		Of course, we have already defined groupoids above, and these are examples of small categories. Indeed, these can be (and often are) defined as small categories such that every element is invertible. We resist that definition here to help mentally separate ``categories'' (the combinatorial input for the definition of the C*-algebras we are interested in) from ``groupoids'' (the topological intermediary between the categories and the C*-algebras). We also, hopefully without confusion, still use $r$ and $d$ for the range and source maps, respectively.
	\end{remark} 
	
	For $\a\in\L$ let $$\a\L = \{\a\b: (\a,\b)\in\L^2\}, \hspace{1cm}\L\a = \{\b\a:(\b,\a)\in\L^2\}.$$ The relation $\a\leq \b \iff \a\in\b\L$ is a preorder on $\L$ and so we can pass to equivalence classes for the relation $\a\approx \b \iff \a\leq\b$ and $\b\leq\a$ and obtain a partial order.  Then (see \cite{OP23}) we have
	\[
	\a\approx\b\iff \a\L = \b\L\iff \a\in\b\L^{-1}\iff\b\in\a\L^{-1}.
	\]
	
		If $\L$ is a LCSC, we say it is {\em finitely aligned} if for all $\a, \b\in \L$ there exists a finite (possibly empty) set $F\subseteq\L$ such that 
		\begin{equation}\label{eq: finitely_aligned}
		\a\L\cap\b\L = \bigcup_{f\in F}f\L
		\end{equation}
	and we say it is {\em singly aligned} or {\em right LCM} if the finite set in \eqref{eq: finitely_aligned} can always be taken to have at most one element.

	Now for a LCSC $\L$ and $\a\in \L$, we also denote by $\a$ the partial bijection $\a: d(\a)\L \to \a\L$ that sends $\b$ to $\a\b$. Then let 
		\[
	S_\L = \text{ inverse semigroup generated by }\{\a\}_{\a\in \L}\text{ inside }\I(\L)
	\] 
		 Note that its inverse in $\I(\L)$ is given by $\a^*:\a\L \to d(\a)\L$ is given by $\a^*(\a\b) = \b$. We call $S_\L$ the {\em left inverse hull of $\L$}. This has also been called the collection of {\em zigzag maps} (see \cite{Sp20} and \cite[Lemma~2.14]{OP20}).
		 
		 By \cite{OP20}[Lemma~2.6] (and \cite[Lemma~3.3]{Sp14} in the inverse-free case), if $\L$ is finitely aligned, then every element $s\in S_\L$ can be written in the form
		 \begin{equation}\label{eq:S_Lam_form}
		 s = \bigvee_{i=1}^{n}\a_i\b_i^* = \bigcup_{i=1}^{n}\a_i\b_i^*
		 \end{equation}
		 for some pairwise compatible set $\{\a_i\b_i^*\}_{i=1}^n$. 
		 \begin{remark}
		 	A couple words of caution are necessary concerning elements of $S_\L$. 
		 	\begin{enumerate}
		 		\item Even if $\{\a_i\b_i^*\}_{i=1}^n$ is a finite compatible set, its join is not necessarily an element of $S_\L$. 
		 		\item Furthermore, even if $\bigvee_{i=1}^n\a_i\b_i^*$ is an element of $S_\L$, it might not be equal to $\bigcup_{i=1}^n\a_i\b_i^*$. This is a crucial point, because multiplication distributes over unions of functions (because $\I(\L)$ is distributive) but not necessarily over general joins. 
		 	\end{enumerate}

		 \end{remark}
		 For the reasons above, we will write a general element of $S_\L$ as a union of functions $\a\b^*$. While $S_\L$ may not be distributive in general, we can still prove and use the following.
		 
		 \begin{lemma}\label{lem:LCSC_dist_ish}
		 	Let $\L$ be a finitely aligned LCSC. 
		 	\begin{enumerate}
		 		\item (Cf. Lemma~\ref{lem:ISG_order_wedge}) Suppose $s = \bigcup_{i=1}^n\a_i\b_i^*\in S_\L$ for $\a_i, \b_i\in\L$, and suppose $e\leqslant a^*a$ for some $a\in S_\L$. Then $ses^* = aea^*$.  
		 		\item (Cf. \cite[Lemma~3.1]{OP20}) If $\xi\subseteq E(S_\L)$ is a tight filter and $\bigcup_{i=1}^n\a_i\a_i^*\in \xi$ for $\a_i\in S_\L$, then $\a_j\a_j^*\in \xi$ for some $1\leq j\leq n$. 
		 	\end{enumerate}
		 \end{lemma}
		 \begin{proof}
		 	 Because multiplication in $S_\L$ distributes over unions, the same proofs as in Lemma~\ref{lem:ISG_order_wedge} and Lemma~\ref{lem:distributive_prime_tight} carry over verbatim. 	 	
		 \end{proof}

		 If $s$ is as in \eqref{eq:S_Lam_form} then we have
		 \[
		 s^*s = \bigcup_{i=1}^{n}\b_i\b_i^* = \Id_{\cup \b_i\L}, \hspace{1cm} ss^* = \bigcup_{i=1}^{n}\a_i\a_i^* = \Id_{\cup \a_i\L}
		 \]
To apply our Uniqueness Theorem~\ref{thm:MainThm1}, we should describe the elements of $S_\L^\Iso$ \eqref{eq:Siso_def}.
	\begin{lemma}\label{lem:SIsodef}
		Let $\L$ be a finitely aligned LCSC. 
		\begin{enumerate}
			\item For $\a, \b\in\L$, we have that $\a \b^*\in S_\L^\Iso$ if and only if \[
			\a\gamma\L \cap \b\gamma\L \neq \emptyset\text{ for all }\gamma\in r(\a)\L.
			\]
			\item For $\a_i, \b_i\in\L$, $i=1, \dots, n$ and $s = \bigcup_{i=1}^{n}\a_i\b_i^*\in S_\L$, we have that $s\in S^\Iso$ if and only if $\a_i\b_i^* \in S_\L^\Iso$ for $ i = 1, \dots, n$. 
		\end{enumerate}
	\end{lemma}
	\begin{proof}
		(1) Suppose $s = \a \b^*\in S^\Iso$, let $\gamma\in r(\a)\L$, and let $ e = \b\ga (\b\ga)^*$; note that $e\leq s^*s$. Then 
		\begin{align}
		0\neq  ses^*e &= \a \b^*\b\gamma(\b\gamma)^* \b \a^*\b\gamma(\b\gamma)^* = \a\Id_{d(\b)\L} \Id_{\gamma\L}\Id_{d(\b)\L}\a^*\Id_{\b\gamma\L}=\a\Id_{\gamma\L}\a^*\Id_{\b\gamma\L} \nonumber\\
		&= \Id_{\a\gamma\L}\Id_{\b\gamma\L}\label{eq:SIso_char}.
		\end{align}
		Because this is nonzero, $\a\gamma\L \cap \b\gamma\L \neq \emptyset$. 
		
		For the converse of (1), suppose that $\a\b^*$ has the given property and take $0\neq e\leqslant s^*s = \Id_{\b\L}$. We can find $\de\in\L$ such that $f:= \Id_{\de\L}\leqslant e$ (as $e$ is a supremum of such elements). This implies $\de\L\subseteq\b\L$ and so $\de = \b\ga$ for some $\delta\in\L$. By assumption, we have $\a\gamma\L \cap \b\gamma\L \neq \emptyset$ so the same calculation as in \eqref{eq:SIso_char} shows that $0\neq sfs^*f\leqslant ses^*e$ (since the product preserves the natural partial order in an inverse semigroup). Thus $s\in S^\Iso$. 
		
		(2) 	 Suppose $s = \bigcup_{i=1}^{n}\a_i\b_i^*\in S_\L^\Iso$ and take $e\leqslant \b_i\a_i^*\a_i\b_i^*$. Then Lemma~\ref{lem:LCSC_dist_ish} implies $\a_i\b_i^*e\b_i\a_i^*  = ses^*\neq 0$, and so $\a_i\b_i^*\in S_\L^\Iso$. 
		 
		 Conversely, suppose $s = \bigcup_{i=1}^{n}\a_i\b_i^*\in S_\L$ and that $\a_i\b_i^*\in S^\Iso$ for all $i$. If $0\neq e\leqslant s^*s = \bigcup_{i=1}^{n}\b_i\b_i^*$, then we can find $i$ such that $\b_i\b_i^* e\neq 0$ and let $f : = \b_i\b_i^* e$. Then by assumption $\a_i\b_i^*f\b_i\a_i^*f\neq 0$, and this element is dominated by $ses^*e$, which therefore cannot be zero. 
	\end{proof}
	
	For the next few proofs we will need to consider the action $\theta$ of $S_\L$ on its tight spectrum, which we will continue to denote $\Et$. To this end, for any subset $X\subseteq \L$ and $\b\in \L$ define
	\[
	\b^{-1}X : = \{\alpha\in \L: \b\a\in X\}.
	\]
	Note that $\b^{-1}(\b\L) = d(\b)\L$. As $S_\L$ is a inverse subsemigroup of $\I(\L)$, each idempotent is of the form $\Id_X$ for some $X\subseteq\L$. Define
	\[
	\J(\L) :=\{X\subseteq \L: \Id_X\in S_\L\}
	\]
	and call this the set of {\em constructible subsets} of $\L$, following Li's definition of constructible right ideals for a semigroup, see \cite[Definition~2.1]{Li13}. It is a semilattice under intersection, and isomorphic to the semilattice $E(S_\L)$. It will be more convenient to describe the action of $S_\L$ on filters in $\J(\L)$. For $X\in\J(\L)$ let $D_X = \{\xi \in \Et : X\in \xi\}$. For an element $\a\b^*\in S_\L$, we have
	\[
	\theta_{\a\b^*}(\xi) = \{\a(\b^{-1}X): X\in \xi\}^{\leqslant}.
	\]
	
	\begin{lemma}\label{lem:iso_intersection}
		Let $\L$ be a finitely aligned LCSC, and suppose that $\alpha\beta^*\in S_\L^\Iso$ with $\alpha\L\cap \beta\L = \bigcup_{i=1}^n\gamma_i\L$. Then $D_{\alpha\L} = D_{\beta\L} = \bigcup_{i=1}^nD_{\gamma_i\L}$. Furthermore, if $\gamma\L\cap \alpha\L \neq \emptyset$ or $\gamma\L\cap \beta\L \neq \emptyset$, then there exists $1\leq i\leq n$ such that $\gamma\L\cap\gamma_i\L\neq \emptyset$.
		
	\end{lemma}
	
	\begin{proof}
		Clearly $\bigcup_{i=1}^nD_{\gamma_i\L}\subseteq D_{\beta\L}$. If $\xi\in D_{\beta\L}$ then $\xi$ is fixed by $\theta_{\alpha\beta^*}$, so $\xi$ contains $\alpha(\beta^{-1}(\beta\L)) = \alpha\L$ and hence also their intersection $\bigcup_{i=1}^n\gamma_i\L$. Thus $\xi\in D_{\bigcup_{i=1}^n\gamma_i\L} = \bigcup_{i=1}^nD_{\gamma_i\L}$ (the last equality by Lemma~\ref{lem:LCSC_dist_ish}). The same argument works for $\alpha$.
		
		If $\gamma\L\cap \beta\L\neq \emptyset$, then any $\xi\in D_{\gamma\L\cap\beta\L}\subseteq D_{\beta\L} = \bigcup_{i=1}^nD_{\gamma_i\L}$ must be in one of the $D_{\gamma_i\L}$, implying that $\gamma_i\L\cap \gamma\L\neq\emptyset$.
	\end{proof}
	
	\begin{definition}
		Let $\L$ be a LCSC and let $x\in \L^0$. A set $B\subseteq x\L$ is called {\em exhaustive (at $x$)} if for every $\a\in x\L$, there exists $\b\in B$ such that $\a\L\cap \b\L\neq \emptyset$. 
	\end{definition}
	
To motivate the next definition, we recall that in \cite{St22} we considered singly aligned LCSCs with one object (i.e. right LCM monoids). If $P$ is such a semigroup, the {\em core} is the subsemigroup $P_c\subseteq P$ such that $p\in P_c\iff pP\cap qP\neq \emptyset$ for all $q\in P$. This semigroup then generates an inverse semigroup $S_c\subseteq S_P$. 
 
 In the finitely aligned LCSC case, the analogous subset of $\L$ will not be a subcategory in general, nor will the analogous subset of $S_\L$ be an inverse subsemigroup. However, if we restrict to elements of $S^\Iso$, we do obtain an inverse subsemigroup.

	\begin{definition}\label{def:FL}
		Let $\L$ be a finitely aligned LCSC. Define $F_\L\subseteq S_\L^\Iso$ by
		\[
		F_\Lambda : = \left\{\bigcup_{i=1}^n{\a_i}{\b_i}^*\in S_\L^\Iso: \{\a_i\}_{i=1}^n\text{ and }\{\b_i\}_{i=1}^n\text{ are exhaustive}\right\}\cup\{0\}.
		\]
	\end{definition}
	\begin{lemma}
		$F_\L$ is an inverse subsemigroup of $S_\L^\Iso$. 
	\end{lemma}
	\begin{proof}
		Clearly $F_\L$ is closed under inverses, so we will be done if we can show it is closed under products. Take $s,t\in F_\L$ and write
		\[
		s = \bigcup_{n=1}^N \a_n\b_n^*, \hspace{1cm} t = \bigcup_{m=1}^M \ga_m\de_m^*.
		\]
		The product will be zero unless there is $x\in \L$ such that $\a_n, \b_n, \ga_m, \de_m\in x\L$ for all $1\leq n\leq N$ and $1\leq m\leq M$. In the case of a nonzero product, exhaustiveness of $\{\ga_m\}$ and $\{\b_n\}$ imply that for each $1\leq n\leq N$ and $1\leq m\leq M$ there exists $K(n,m)\geq 1$ and $\tau_k^{(n,m)}\in \Lambda$ for each $1\leq k\leq K(n,m)$ such that
		\begin{equation}\label{eq:bn_gm}
		\b_n \L \cap \ga_m\L = \bigcup_{k=1}^{K(n,m)} \tau_k^{(n,m)}\L
		\end{equation}
		so that, for each  $1\leq k\leq K(n,m)$ we have
		\begin{equation}\label{eq:tau_b_ga}
		\tau_k^{(n,m)} = \b_n\b_{n,m,k} = \ga_m\ga_{n,m,k}
		\end{equation}
		for elements $\b_{n,m,k}, \ga_{n,m,k}\in \L$. Then one has (by distributativity in the symmetric inverse monoid)
		\[
		\a_n\b_n^*\ga_m\de^*_n = \a_n\left(\bigcup_{k=1}^{K(n,m)}\b_{n,m,k}\ga_{n,m,k}^*\right)\de_m^* = \bigcup_{k=1}^{K(n,m)}\a_n\b_{n,m,k}(\de_m\ga_{n,m,k})^* .
		\]
		So again by distributativity in the symmetric inverse monoid we have
		\[
		st = \bigcup_{n=1}^N\bigcup_{m=1}^M\bigcup_{k=1}^{K(n,m)}\a_n\b_{n,m,k}(\de_m\ga_{n,m,k})^*.
		\]
		We will be done if we can show that $\{\a_n\b_{n,m,k}\}$ and $\{\de_m\ga_{n,m,k}\}$ are exhaustive (at $x$).
		
		Let $\eta\in x\L$ and find $\b_n$ such that $\b_n\L\cap \eta\L\neq\emptyset$. By Lemma~\ref{lem:iso_intersection} we can find $\rho\in \a_n\L\cap\b_n\L$ such that $\eta\L\cap \rho\L\neq \emptyset$. Write $\rho = \a_n\a' = \b_n\b'$ for $\a',\b'\in\Lambda$ and find $\eta', \rho''\in\Lambda$ such that $\eta\eta' = \rho\rho''$. 
		
		The element $\b_n\alpha'\rho''$ is well-defined and has range $x$, so we can find $m$ such that $\ga_m\L\cap \b_n\alpha'\rho''\L\neq\emptyset$. If $\zeta\in \ga_m\L\cap \b_n\alpha'\rho''\L\neq\emptyset$ we can, by \eqref{eq:bn_gm} find $\tau', \ga',\b''\in\L$ and $1\leq k\leq K(n,m)$ such that
		\[
		\zeta = \ga_m\ga' = \b_n\alpha'\rho''\b'' = \tau_k^{(n,m)}\tau' \stackrel{\eqref{eq:tau_b_ga}}{=} \b_n\b_{n,m,k}\tau'.
		\]
		Thus left cancellativity implies $\b_{n,m,k}\tau' = \alpha'\rho''\b''$. And so we have
		\[
		\a_n\b_{n,m,k}\tau' = \a_n\alpha'\rho''\b'' = \rho\rho''\b'' = \eta\eta'
		\]
		implying that $\eta\Lambda\cap \a_n\b_{n,m,k}\Lambda\neq\emptyset$. Since $\eta\in x\L$ was arbitrary, $\{\a_n\b_{n,m,k}\}$ is exhaustive. Applying inverses and running the same argument gives that $\{\de_m\ga_{n,m,k}\}$ is also exhaustive, and so $F_\L$ is closed under products.
	\end{proof}

	\begin{lemma}\label{lem:Iso_to_core}
		Suppose that $\alpha\beta^*\in S_\L^\Iso$ and that $\de\L\subseteq \a\L \cap \b\L$.  Then we have that $\de^{*}\a\b^{*}\de = \bigcup_{i=1}^n \de_i\a_i^{*}$ for exhaustive sets  $\{\de_i\}_{i=1}^n$ and $\{\a_i\}_{i=1}^n$, and hence $\de^*\a\b^*\de\in F_\L$. 
	\end{lemma}
	\begin{proof}
		Write $\delta  = \beta\beta_1$ for some $\beta_1 \in d(\alpha)\L$, and let $\xi\in D_{\delta\L}$ be any ultrafilter. We have $D_{\delta\L}\subseteq D_{\beta\L}$, and so $\theta_{\a\b^*}(\xi) = \xi$. This implies $\alpha(\beta^{-1}(\delta\L)) = \alpha(\beta^{-1}(\beta\beta_1\L)) = \alpha\beta_1\L\in \xi$, and so has nonempty intersection with $\delta\L$.
		Write
			\[
			\delta\L \cap \alpha\beta_1\L = \bigcup_{i=1}^n\gamma_i\L.
			\]
			with $\gamma_i = \delta\delta_i = \alpha\beta_1\alpha_i$ for some $\alpha_i, \delta_i\in\L, i = 1, \dots, n$. We have
			\begin{align*}
				\de^*\a\b^*\de & = \de^*\a\b^*\b\b_1\\
				& = \de^*\a\b_1\\
				& = \de^*\left(\bigcup_{i=1}^n\gamma_i\gamma_i^*\right) \a\b_1\\
				&=\bigcup_{i=1}^n \de_i\a_i^*.
			\end{align*}
			Let $z\in d(\delta)\L$ and let $\xi\in D_{\delta z\L}\subseteq D_{\delta\L}\subseteq D_{\beta\L}$ be any tight filter. As above, $\xi$ is fixed by $\theta_{\a\b^*}$, so $\xi$ contains $\de\L$ and $\a(\b^{-1}(\de\L)) = \a\b_1\L$, implying that their intersection $\bigcup_{i=1}^n\gamma_i\L$ is in $\xi$. Tightness implies that $\gamma_i\L = \delta\delta_i\L$ is in $\xi$ for some $1\leq i\leq n$ (by Lemma~\ref{lem:distributive_prime_tight}). Since $\delta\delta_i\L$ and $\delta z\L$ are both in $\xi$ they must have nonempty intersection, implying that $z\L \cap \delta_i\L \neq\emptyset$. A similar calculation shows that $z\L\cap \alpha_j\L\neq \emptyset$ for some $1\leq j\leq n$.
			
			To conclude that $\de^*\a\b^*\de\in F_\L$ we are left to show that the product is in $S_\L^\Iso$. Because $\theta_{\a\b^*}$ fixes every point in its domain, the same is true of $\theta_{\de^*\a\b^*\de}= \theta_\de^{-1}\circ \theta_{\a\b^*}\circ\theta_\de$, and so by \cite[Lemma~4.9]{EP16}, $\de^*\a\b^*\de\in S^\Iso_\L$. 
	\end{proof}
	
	\begin{thm}\label{thm:main_theorem}
		Let $\L$ be a countable finitely aligned left cancelative small category, let $S_\L$ be its associated inverse semigroup, and let $F_\L\subseteq S_\L^\Iso$ be as in Definition~\ref{def:FL} and Lemma~\ref{lem:SIsodef}. Let $\{T_s:s\in S_\L\}$ denote the canonical generating set of $C^*_r(\gt(S_\L))$ and let 
		\begin{itemize}
			\item $\Q_{r}^\Iso(\L)$ be the subalgebra of $C^*_r(\gt(S_\L))$ generated by $\{T_s:s\in S^\Iso_\L\}$, and
			\item $\Q^\Iso_{r,c}(\L)$ be the subalgebra of $C^*_r(\gt(S_\L))$ generated by $\{T_s:s\in F_\L\}$.
		\end{itemize}
		Then both $\Q_{r}^\Iso(\L)$ and $\Q^\Iso_{r,c}(\L)$ are essentially ideal-detecting.
	\end{thm}
	
	\begin{proof} Theorem~\ref{thm:MainThm1} implies $\Q^\Iso_{r}(\L)$ is essentially ideal-detecting. 
		
		Suppose that $B$ is a C*-algebra and that $\pi: C^*_r(\gt(S_\L))\to B$ is a $*$-homomorphism such that is injective on $\Q^\Iso_{r,c}(\L)$. By Theorem~\ref{thm:MainThm1}, we will be done if we can show that $\ker(\pi)\cap\Q^\Iso_{r}(\L)\subseteq \Js(\gt(S_\L))$. 
		
		As in the corrigendum to \cite{St22}, we will show that \begin{equation}\label{eq:main_thm_contraction}
			\|\pi(a)\| \geq \|\Ered(a)\|
		\end{equation}for all $a$ in the *-algebra $A_0$ generated by the set $\{T_{s}: s\in S_\L^\Iso\}$.
			As in the corrigendum \cite{BS24} to \cite{St22} we can take a finite linear combination in $A_0$:
			\begin{equation}
			a = \sum_{f\in F}\lambda_fT_{s_f}.
			\end{equation}
		and so by \cite[Proposition~3.14]{EP16} one representative of $\Ered(a)$ is given by $\Ered(a) = \sum_{f\in F}\lambda_f1_{\cF_{s_f}}$ for certain compact open sets $\cF_{s_f}\subseteq D_{\dom(s_f^*s_f)}$.
		
		Each nonempty $F'\subseteq F$ determines a subset of $\Et$ given by
		\[
		U_{F'} := \bigcap_{f\in F'}D_{\dom(s_f^*s_f)} \setminus \left(\bigcup_{g\in F\setminus F'} D_{\dom(s_g^*s_g)}\right)
		\] 
		 Since $\bigcup_{f\in F}D_{\dom(s^*_fs_f)}=\bigsqcup_{\emptyset\neq F'\subseteq F}U_{F'}$ and $\cF_{s_f}\subseteq D_{\dom(s^*_fs_f)}$ for all $f\in F$, the support of $E(a)$ is contained in $\bigsqcup_{\emptyset\neq F'\subseteq F}U_{F'}$. 
		
		As we are trying to show \eqref{eq:main_thm_contraction}, we need only consider cases where $\Ered(a)\neq 0$, so that there exists $\lambda\in \CC\setminus\{0\}$ such that the function $a$ takes the value $\lambda$ on a subset of $\Et$ with nonempty interior and that $\|\Ered(a)\| = |\lambda|$. Since $a$ takes only finitely many values and is constant on each $U_{F'}$, we can find $\emptyset\neq F'\subseteq F$ such that $a|_V =  \lambda$ for an open subset $V\subseteq U_{F'}$. Since $V$ is open, it contains an ultrafilter $\xi$. 
		
		For each $f\in F$, write $s_f = \bigcup_{i=1}^{n(f)}\a^f_i\b_i^{f*}$. Since $\xi\in V \subseteq U_{F'}$, $\xi\in D_{\dom(s_fs_f^*)} = D_{\cup\b^f_i\L} = \cup D_{\b^f_i}$ (the last equality following from Lemma~\ref{lem:LCSC_dist_ish}), and since $s_f\in S_\L^\Iso$, Lemma~\ref{lem:iso_intersection} and Lemma~\ref{lem:SIsodef} imply that there exists $k(f)$ with $1\leq k(f)\leq n(f)$ such that $\xi\in D_{\b^f_{k(f)}\L} = D_{\a^f_{k(f)}\L}$, i.e. $\a^f_{k(f)},\b^f_{k(f)}\L\in \xi$. 
		
		Because $\xi\notin D_{\dom(s_g^*s_g)}$ for all $g\in (F')^c$, Lemma~\ref{lem:SIsodef} and Lemma~\ref{lem:iso_intersection} again imply that $\a_i^g\L, \b_i^g\L\notin \xi$ for all $g\in (F')^c$ and $1\leq i\leq n(g)$. 
		
		Because $\xi$ is an ultrafilter, for each $g\in (F')^c$ we can find $A, B\in \xi$ such that $A\cap (\cup \a_i^g\L) = \emptyset$, and $B\cap (\cup \b_i^g\L) = \emptyset$. Let $e_g = A\cap B \in \xi$. 
		
		Since $e_g\in \J(\L)$ we can write $e_g$ as a finite union of principal right ideals, so Lemma~\ref{lem:LCSC_dist_ish} gives us that we can find $\tau_g\in\L$ with $\tau_g\L\subseteq e_g$ and $\tau_g\L\in\xi$. 
		
		Because all the elements found above are in $\xi$, we must have 
		\[
		\left(\bigcap_{g\in (F')^c}\tau_g\L\right)\cap\left(\bigcap_{f\in F'}\a_i^f\L\right)\cap \left(\bigcap_{f\in F'}\b_i^f\L\right) \neq \emptyset
		\]
		so finite alignment implies it is equal to a finite union of principal right ideals. Pick any one of them and call it $\sigma\L$. 
		
		We have $D_{\sigma\L}\subseteq U_{F'}$. Since $\|\Ered(a)\|=|\Ered(a)(\xi)|$, $\xi\in D_{\sigma\L}$, and $T_{\sigma\sigma^*}=1_{D_{\sigma\L}}$, we have 
		\begin{equation}
			\label{eqn:norm}
			\|\Ered(a)\| =\|T_{\sigma\sigma^*}\Ered(a)\|
		\end{equation} 
		
		We claim that $\sigma^*s_f\sigma\in F_\L$ for all $f\in F'$ and $\sigma^*s_g\sigma = 0$ for all $g\in (F')^c$. 
		
		Because $\sigma\sigma^*\leqslant \b_i^f\b_i^{f*}$, Lemma~\ref{lem:ISG_order_wedge} implies $\sigma^*s_f\sigma = \sigma^*\a_i^f\b_i^{f*}\sigma$, which is in $F_\L$ by Lemma~\ref{lem:Iso_to_core}.
		
		 We also have that $\sigma\sigma^*s_g^*s_g = \sigma\sigma^*\left(\bigcup \b_i^g\b_i^{g*}\right) = 0$ because $\sigma\L\subseteq \tau_g\L$ which was chosen to be disjoint from $\cup\b_i^g\L$. Hence $\sigma^*s_g\sigma = \sigma^*s_g(s_g^*s_g\sigma\sigma^*)\sigma = 0$. 
		 
		 So we calculate 
		\begin{align*}
			\|\pi(a)\| &= \left\|\pi\left(\sum_{f\in F}\lambda_fT_{s_f}\right)\right\|\\
			&\geq \left\|\pi(T_{\sigma}^*)\pi\left(\sum_{f\in F}\lambda_fT_{s_f}\right)\pi(T_{\sigma})\right\|&\text{$\pi(T_{\sigma})$ a partial isometry}\\
			&= \left\|\pi\left(\sum_{f\in F'}\lambda_fT_{\sigma}^*T_{s_f}T_{\sigma}\right)\right\|&\text{by choice of $\sigma$}\\
			& = \left\|\sum_{f\in F'}\lambda_fT_{\sigma}^*T_{s_f}T_{\sigma}\right\|&\text{$\pi$ is isometric on $\Q^\Iso_{r,c}(\L)$}\\
			&\geq  \left\|\sum_{f\in F'}\lambda_fT_{\sigma\sigma^*}T_{s_f}T_{\sigma\sigma^*}\right\|&\text{submultiplicativity}\\
			&\geq  \left\|\Ered\left(\sum_{f\in F'}\lambda_fT_{\sigma\sigma^*}T_{s_f}T_{\sigma\sigma^*}\right)\right\|&\text{$\Ered$ is contractive}\\
				&=  \left\|\Ered\left(\sum_{f\in F'}\lambda_fT_{\sigma\sigma^*}T_{s_f}T_{\sigma\sigma^*}+ \sum_{g\in (F')^c}\lambda_gT_{\sigma\sigma^*}T_{s_g}T_{\sigma\sigma^*} \right)\right\|&\text{adding }0\\
			&=\left\|T_{\sigma\sigma^*}\Ered(a)T_{\sigma\sigma^*}\right\|&\text{\cite[Lemma~3.2]{KM21}}\\
			& = \|\Ered(a)\| & \text{by \eqref{eqn:norm}}
		\end{align*}

		 Thus $\pi(a) = \pi(b) \implies \|\pi(a-b)\| = 0 \implies \|\Ered(a-b)\| = 0 \implies \Ered(a) = \Ered(b)$, so $\pi(a)\mapsto \Ered(a)$ is a well-defined linear contraction on $\pi(A_0)$, and so extends to a linear contraction on $\pi(\Q^\Iso_{r}(\L))$. 
		Suppose $a\in\ker(\pi)\cap \Q^\Iso_{r}(\L)$. Then $\|\Ered(a^*a)\| = 0$.
		 By \eqref{JsingDef}, we have $a^*a\in \Js(\gt(S_\L))$, implying $a\in \Js(\gt(S))$.
	
	\end{proof}

	\subsection{The singly aligned case}
	
	We now specialize to the case where $\L$ is a singly aligned LCSC. 
	
	Let $\L$ be a singly aligned LCSC. Then the constructible subsets are given by
	\[
	\J(\L) = \{\a\L: \a\in \L\}\cup\{\emptyset\}.
	\] 
	\begin{lemma}\label{lem:hull_form}
		Let $\L$ be a singly aligned LCSC. Then
		\begin{equation}\label{eq:hull}
		S_\L = \{\a \b^*: d(\a) = d(\b)\}\cup\{0\}.
		\end{equation}
		Furthermore, we have $\a \b^* = \gamma \tau^*$ if and only if there exists an invertible element $u\in\L^{-1}$ such that $\a u = \gamma$ and $\b u = \tau$. The product in $S_\L$ is given by
		\[
		\a\b^*\ga\tau^* = \begin{cases}
			\a\gamma_1(\tau\b_1)^*& \b\L\cap \gamma\L = \rho\L\text{ with }\rho = \b\b_1 = \gamma\gamma_1\\
			0&\text{otherwise}
		\end{cases}
		\]
		and all products involving 0 equal to 0. 
\end{lemma} 
\begin{proof}
	For \eqref{eq:hull}, we note that the containment $\supseteq$ is clear. For the other, suppose $\a^*\b \neq 0$. Then $\b\L\cap \a\L \neq \emptyset$, so there exists $\gamma\in\L$ such that $\gamma\L= \a\L\cap \b\L$. So $\gamma = \a\a_1 = \b\b_1$, and thus
	\[
	\a^*\b = \a^* \Id_{\a\L}\Id_{\b\L}\b =  \a^*\Id_{\gamma\L}\b = \a^*{\a\a_1}{\b\b_1}^*{\b} = \Id_{d(\a)\L}{\a_1}{\b_1}^*\Id_{d(\b)} = {\a_1}{\b_1}^*.
	\]
	Hence \eqref{eq:hull} is closed under products and is clearly closed under inverses, and so is an inverse semigroup containing each $\a$. Thus we have that $\subseteq$ containment and the equality is proven.
	
	Now suppose $\a \b^* = \gamma \tau^*$. The domains of these maps must be equal implying $\b\L = \tau\L$, and so there is an invertible element $u$ such that $\b u = \tau$. Applying these maps to the element $\tau$ then gives
	\[
	\gamma = \gamma \tau^*(\tau) = \a \b^*(\b u) = {\a}(u) = \a u.
	\]
	Conversely, suppose there is an invertible element $u$ with $\a u = \gamma$ and $\b u = \tau$. It is clear that $uu^* = \Id_{u\L}= \Id_{d(u)\L} = \Id_{d\tau)\L} = \Id_{d(\gamma)\L}$. Thus
	\[
	\a \b^* = \gamma uu^* \tau^* = \gamma \Id_{d(\gamma)\L}\tau^* = \tau \gamma^*.
	\]
\end{proof}
The idempotent semilattice $E(S_\L)$ is given by
\[
E(S_\L) = \{\a \a^*:\a\in\L\}\cup\{0\}.
\]
	and the map that sends $\a \a^*$ to $\a\L$ is an isomorphism of semilattices from $E(S_\L)$ to $\J(\L)$. 

\begin{definition}
	Let $\L$ be a small category. The {\em core} of $\L$, denoted $\L_c$, is defined to be
	\[
	\L_c = \{\a\in\L: \a\L\cap \b\L \neq \emptyset\text{ for all }\b\in r(\a)\L\}.
	\]
\end{definition}
\begin{lemma}\label{lem:core_properties}
	Let $\L$ be a LCSC. The core $\L_c$ is a subcategory of $\L$ containing the invertible elements. Furthermore, 
	\begin{enumerate}
		\item $\a\b\in \L_c \implies \a,\b\in \L_c$, 
		\item For $\a,\b\in \L_c$, $\a\L\cap \b\L = \bigcup_{i=1}^n\a_i\L$ implies that $\{\a_i\}$ is exhaustive.
		\item If $\L$ is singly aligned, then $\a\L\cap \b\L = \ga\L$ implies $\ga\in\L_c$. 
	\end{enumerate}
\end{lemma}

\begin{proof}
	If $\ga$ is invertible and $r(\a) = r(\ga)$, then $\a = \ga\ga^{-1}\a \in \a\L\cap \gamma\L$. 
	
	Suppose that $\a,\beta\in \L_c$ with $d(\a) = r(\beta)$, and take $\gamma\in r(\a)\L$. Then $\a\a' = \gamma\gamma'$ for some $\a',\ga'\in \L$. Since $r(\a') = d(\a) = r(\b)$, we have that $\b\b' = \a'\a''$ for some $\b', \a''\in\L$. Thus
	\[
	\a\b\b' = \a\a'\a'' = \ga\ga'\a'' \in \a\b\L\cap \ga\L.
	\]
	Since $\gamma$ was arbitrary, $\a\b\in \L_c$.
	
	(1) Suppose $\a\b\in \L_c$. Clearly this implies $\a\in \L_c$, as $\a\b\L\cap \ga\L\subseteq \a\L\cap \gamma\L$. Now if $\delta\in d(\a)\L = r(\b)\L$, we have that $\a\de\L\cap \a\b\L \neq \emptyset$, so that left cancelativity implies that $\de\L\cap \b\L\neq \emptyset$. 
	
	(2) Suppose $\delta\in r(\gamma)\L$. Since $r(\gamma) = r(\a)$, $\a\in\L_c$ implies $\a\a' = \de\de'$ for some $\a', \de'\in \L$. Now $r(\a\a') = r(\b)$ and $\b\in\L_c$ implies $\b\b' = \a\a'\a''$ for some $\b', \a''\in\L$. Thus we have $\b\b' = \a\a'\a'' = \de\de'\a''$,  an element in $\a\L\cap\b\L\cap\de\L = \left(\bigcup_i\a_i\L\right)\cap \de\L$, implying that $\{\a_i\}$ is exhaustive. (3) now follows from (2).
	
\end{proof}
Note that Lemma~\ref{lem:core_properties} together with the form of the product given in Lemma~\ref{lem:hull_form} imply that the set $S_c\subseteq S_\L$ given by
\begin{equation}\label{eq:S_cDef}
	S_c:=\{\a\b^*\in S_\L : \a, \b\in\L_c\}\cup\{0\}
\end{equation}
is a inverse subsemigroup of $S_\L$. One can exclude the 0 element and still get a inverse subsemigroup if and only if $\L^0$ is a singleton (i.e. when $\L$ is a monoid). As before we have
\begin{align*}
	F_\L &= \{\a\b^*\in S_\L: \a,\b\in \L_c, \a\ga\L\cap\b\ga\L\neq \emptyset\text{ for all }\ga\in d(\b)\L\}\\
	& = S_\L^\Iso\cap S_c
\end{align*}
\begin{remark}
	In the finitely (non-singly) aligned case, one would not expect that the set analogous to \eqref{eq:S_cDef} would be closed under products. It seems to be a special feature to the singly aligned case that the core generates an inverse subsemigroup.
\end{remark}
The following theorem now follows as a corollary of Theorem~\ref{thm:main_theorem}
	\begin{thm}\label{thm:main_theorem_sa}
		Let $\L$ be a countable singly aligned left cancelative small category and let $\L_c$ be the core of $\L$. Let $\{T_{\a\b^*}:\a,\b\in\L\}$ denote the canonical generating set of $C^*_r(\gt(S_\L))$ and let 
		\begin{itemize}
			\item $\Q_{r,c}(\L)$ be the subalgebra of $C^*_r(\gt(S_\L))$ generated by $\{T_{\a}:\a \in\L_c\}$,
			\item $\Q_{r}^\Iso(\L)$ be the subalgebra of $C^*_r(\gt(S_\L))$ generated by $\{T_{\a\b^*}:\a\b^*\in S_\L^\Iso\}$, 	 and
		\item $\Q_{r,c}^\Iso(\L)$ be the subalgebra of $C^*_r(\gt(S_\L))$ generated by $\{T_{\a\b^*}:\a\b^*\in S_\L^\Iso\cap S_c\}$.
		\end{itemize}
		Then each of $\Q_{r,c}(\L)$, $\Q_{r}^\Iso(\L)$, and $\Q_{r,c}^\Iso(\L)$ is essentially ideal-detecting.
	\end{thm}
	
	\begin{remark}\label{rmk:CK_LCSC}
	Spielberg defines $\cO(\L)$, the {\em Cuntz-Krieger algebra} of a finitely aligned LCSC as the universal C*-algebra of a groupoid $G_2$ restricted to a subspace $\partial\L$ of its unit space. Spielberg shows shows it is the universal C*-algebra for generators $\{W_\a:\a\in\L\}$ subject to the relations
	\begin{enumerate}
		\item[(S1)]\label{it1} $W_\a^*W_\a = W_{d(\a)}$.
		\item[(S2)]\label{it2} $W_\a W_\b = W_{\a\b}$ if $d(\a) = r(\b)$.
		\item[(S3)]\label{it3} $W_\a W_\a^* W_\b W_\b^* = \bigvee_{i = 1}^nW_{\ga_i} W_{\ga_i}^*$ if $\a\L\cap \b\L = \cup_{i=1}^n\ga_i\L$.
		\item[(S4)]\label{it4} $W_x = \bigvee_{\a\in F} W_\a W_\a^*$ if $x\in \L^0$ and $F$ is exhaustive at $x$. 
	\end{enumerate}
	By \cite{LiGarside1}, $\cO(\L)$ is the universal C*-algebra for the tight groupoid of $S_\L$. Upon making some assumptions on this groupoid we can formulate a uniqueness theorem akin to the original Cuntz-Krieger uniqueness theorem (where injectivity is equivalent to injectivity on a subalgebra generated by some subset of the generators). 
		\end{remark}
\begin{cor}
	Let $\L$ be a singly aligned left cancellative small category and let $\L_c$ be the core of $\L$. Suppose the tight groupoid of $S_\L$ is amenable and has zero singular ideal. Then if $\vp:\cO(\L)\to B$ is a $*$-homomorphism into a C*-algebra $B$, $\vp$ is injective if and only if it is injective on the subalgebra generated by $\{W_\a:\a\in\L_c\}$. 
\end{cor}	

We close this section by applying our results to right LCM monoids. The C*-algebras of such monoids have been studied by many authors, see \cite{ABLS19, BOS18, BLS16, BLS18, BS16, LL20, LaLi22, LiB19, NS22, Stam15, Stam17, StLCM}. In the language of this paper so far, a right LCM monoid is a singly aligned LCSC with only one object. With Theorem~\ref{thm:main_theorem_sa} in hand, we can now prove the following generalization of \cite[Theorem~4.1]{St22} (later, correctly proven in \cite{BS24}) to the possibly non-Hausdorff case.

\begin{cor}\label{cor:LCM} (Cf. \cite[Theorem~4.1]{St22}, \cite[Theorem~2.1]{BS24})
	Let $P$ be a right LCM monoid, let $\Q_r(P) = C^*_r(\gt(S_P))$ be its reduced boundary quotient, and let $\Q_{r,c}(P)$ be the subalgebra generated by the core submonoid. Then $\Q_{r,c}(P)$ is essentially ideal-detecting.
\end{cor}

	\section{$k$-graphs}
	
	A {\em $k$-graph} is a pair $(\L, \deg)$ where $\L$ is a small category and  $\deg:\L\to \NN^k$ (the {\em degree map}) satisfies
	\begin{itemize}
		\item $\deg(\a\b) = \deg(\a) + \deg(\b)$ for all $\a,\b\in\L$ with $r(\b) = d(\a)$, and
		\item if $\deg(\a) = m+n$ for some $m,n\in \NN^k$, there exist unique $\b,\ga\in \L$ with $\a = \b\ga$, $\deg(\b) = m$, and $\deg(\ga) = n$ (this is called the {\em unique factorization property}). 
	\end{itemize} 
	One can see that the unique factorization property implies that $\L$ is left and right cancellative, and is also has no invertible elements.

	In \cite{BNR14}, they define what it means for  $(\a,\b)\in\L\times\L$ to be a {\em cycline pair}. In our notation, this is equivalent to saying $d(\a) = d(\b)$ and $D_{\b\ga\L}= D_{\a\ga\L}$ for all $\ga\in d(\a)\L$, which one can easily see is equivalent to saying $\a\b^*\in S^\Iso_\L$.
	
	If one assumes that $\L$ is finitely aligned, then this combined with right cancellation implies that $\gt(S_\L)$ is Hausdorff \cite[Lemma~7.1]{Sp20}. The map $\a\b^*\mapsto \deg(\a)-\deg(\b)$ induces a cocycle $\phi:\gt(S_\L)\to\ZZ^k$ such that $\phi^{-1}(0)$ is an AF-groupoid (see \cite[Section~3]{RSY04}), and so $\gt(S_\L)$ is amenable by \cite[Corollary~4.5]{RW17}. Thus we can apply our results to obtain the following generalization of the main theorem in \cite[Theorem~7.1]{BNR14}.
	
	\begin{thm}\label{thm:k_graph}
		Let $\L$ be a finitely aligned $k$-graph, let $\cO(\L)$ be the Cuntz-Krieger algebra of $\L$, and suppose $\vp:\cO(\L)\to B$ is a $*$-homomorphism into a C*-algebra $B$. Then $\vp$ is injective if and only if it is injective on the subalgebra $C^*(\{W_\a W_\b^*: (\a,\b)\text{ is cycline}\})$. 
	\end{thm}
	
		\begin{proof}
		By the discussion before the theorem, $\cO(\L)\cong C_r^*(\gt(S_\L))$ with the isomorphism sending $W_\a$ to $T_\a$. To apply Theorem~\ref{thm:main_theorem}, we will show that the given subalgebra coincides with $\Q_{r}^\Iso(\L)$. Let $s \in S_\L^\Iso$, and write $s = \bigcup_{i=1}^n\a_i\b_i^*$ for $\a_i, \b_i\in\L$. By work in \cite[Section~3]{DM14} following \cite[Theorem~6.3]{Sp14}, we have that $T_s = \bigvee_{i=1}^nT_{\a_i}T_{\b_i^*}$, where this wedge in a C*-algebra is as in \eqref{eq:operator_s_vee_t}. By that formula and Lemma~\ref{lem:SIsodef} we see $T_s$ is in the C*-algebra generated by the set $\{T_\a T_\b^*: \a\b^*\in S_\L^\Iso\}$. 
	\end{proof}




\bibliographystyle{alpha}
\bibliography{C:/Users/slear/Dropbox/Research/bibtex}{}
	
\end{document}